\documentclass[twoside,11pt,reqno]{amsart}
\usepackage{amsmath,amssymb,amscd,mathrsfs,amscd, todonotes}
\usepackage{graphics,verbatim}
\usepackage{todonotes}
\usepackage{hyperref}

\newcommand{\losemi}{{\otimes \kern -.78em \ltimes}}
\newcommand{\rosemi}{{\otimes \kern -.78em \rtimes}}

\makeatletter
\newcommand{\leqnomode}{\tagsleft@true}
\newcommand{\reqnomode}{\tagsleft@false}
\makeatother

\newtheorem{theorem}{Theorem}[subsection]

\makeatletter\let\c@fact\c@theorem\makeatother

\makeatletter\let\c@note\c@theorem\makeatother

\newtheorem{lemma}{Lemma}[subsection]
\makeatletter\let\c@lemma\c@theorem\makeatother

\makeatletter\let\c@lemma\c@theorem\makeatother

\newtheorem{quest}{Question}[subsection]
\makeatletter\let\c@alg\c@theorem\makeatother

\newtheorem{prop}{Proposition}[subsection]
\makeatletter\let\c@prop\c@theorem\makeatother

\makeatletter\let\c@conj\c@theorem\makeatother

\newtheorem{cor}{Corollary}[subsection]
\makeatletter\let\c@cor\c@theorem\makeatother

\makeatletter\let\c@defn\c@theorem\makeatother

\theoremstyle{definition}

\newtheorem{remark}{Remark}[subsection]
\makeatletter\let\c@remark\c@theorem\makeatother

\makeatletter\let\c@example\c@theorem\makeatother
\numberwithin{equation}{subsection}

\usepackage[capitalise]{cleveref}
\crefname{theorem}{Theorem}{Theorems}
\crefname{fact}{Fact}{Facts}
\crefname{note}{Note}{Notes}
\crefname{lemma}{Lemma}{Lemmas}
\crefname{alg}{Algorithm}{Algorithms}
\crefname{remark}{Remark}{Remarks}
\crefname{example}{Example}{Examples}
\crefname{prop}{Proposition}{Propositions}
\crefname{conj}{Conjecture}{Conjectures}
\crefname{cor}{Corollary}{Corollaries}
\crefname{defn}{Definition}{Definitions}
\crefname{equation}{\!\!}{\!\!} 

\newcounter{listequation}

\begin{document}
\title{On the existence of mock injective modules for algebraic groups}

\author{William D. Hardesty}
\address{Department of Mathematics \\
          University of Georgia \\
          Athens, GA 30602}
\email{hardes1@uga.edu}
\thanks{Research of the first author was partially supported by NSF RTG grant  DMS-1344994}
\author{Daniel K. Nakano}
\address{Department of Mathematics \\
          University of Georgia \\
          Athens, GA 30602}
\thanks{Research of the second author was partially supported by NSF grants
DMS-1402271 and DMS-1701768}\
\email{nakano@math.uga.edu}
\author{Paul Sobaje}
\address{Department of Mathematics \\
          University of Georgia \\
          Athens, GA 30602}
\email{sobaje@uga.edu}
\thanks{Research of the third author was partially supported by NSF RTG grant  DMS-1344994}
\date{\today}

\begin{abstract} Let $G$ be an affine algebraic group scheme over an algebraically closed field $k$ of characteristic $p>0$, and let $G_r$ denote the $r$-th Frobenius kernel of $G$.  
Motivated by recent work of Friedlander, the authors investigate the class of mock injective $G$-modules, which are defined to be those rational $G$-modules that are injective on restriction to $G_r$ for all $r\geq 1$.  
In this paper the authors provide necessary and sufficient conditions for the existence of non-injective mock injective $G$-modules, thereby answering a question raised by Friedlander. 
Furthermore, the authors investigate the existence of non-injective mock injectives with simple socles. Interesting cases are discovered that show that this can occur for reductive groups, but 
will not occur for their Borel subgroups. 
\end{abstract}

\maketitle

\section{Introduction}

\subsection{}
Let $k$ be an algebraically closed field and $G$ be an affine algebraic group scheme over $k$, defined over the prime subfield $\mathbb{F}_p$.  The $\mathbb{F}_p$-structure on $G$ admits 
a Frobenius morphism $F:G \rightarrow G$.  Let $G_r$ denote the kernel of the $r$th iteration of the Frobenius map, $F^r$, and let $G(\mathbb{F}_q)$ denote the $F^r$-fixed point subgroup of $G(k)$, where $q=p^r$.  A fundamental problem in the representation theory of $G$ is to determine the extent to which the finite subgroup schemes $G_r$ and $G(\mathbb{F}_q)$ detect the cohomology of $G$.  For semi-simple groups, the relationship between $G$, $G_r$, and $G(\mathbb{F}_q)$ cohomology has been studied extensively (see \cite{And1,And2}, \cite{BNP1, BNP2, BNP3, BNP4, BNP5}, \cite{CPS}, 
\cite{CPSvdK}, and \cite{N} for a survey).

Recently this problem has arisen in a new context.  In \cite{F1}, Friedlander defined a support theory for rational $G$-modules when $G$ has a particular ``exponential" structure, a condition met by many important classes of groups.  The aforementioned work is a generalization of the support variety theory that has been developed over the last several decades for finite group schemes and related algebraic objects, of which $G_r$ and $G(\mathbb{F}_q)$ are particular examples.  Among the many important features of support varieties in the finite group scheme setting is their ability to detect the injectivity of a module (which is equivalent to detecting projectivity in this case).  A natural question in \cite{F1, F2} that is investigated is whether or not the newly introduced rational supports 
detect injectivity for rational $G$-modules.  

In \cite{F2}, it is shown that a $G$-module $M$ has trivial rational support precisely when $M$ is injective over every Frobenius kernel of $G$.  Such modules are dubbed to be ``mock injective.''  Friedlander then proved that mock injectivity is in general a weaker condition than injectivity.  In particular, he established the existence of non-injective mock injective modules for $G$ in cases when $G$ is not a reductive group using the following argument.  Embed $G$ into $GL_n$ for some $n$.  The coordinate algebra $k[GL_n]$ is injective over every finite subgroup scheme of $GL_n$, hence over $G_r$ for every $r$.  On the other hand, $k[GL_n]$ is injective over the reduced subgroup scheme $G$ if and only if $G$ is reductive.  
Consequently, there exists non-injective mock injective $G$-modules.  Such modules are referred to as ``proper mock injective modules.'' The example above leads to a clear question: for arbitrary $G$, do there exist proper mock injective modules?  In this paper we provide a definitive answer to this question via the construction of new families of examples of proper mock injective modules.

\subsection{} A key tool in our study of injective $G$-modules is the induction functor \cite[\S I.3]{J}.  If $H \le G$, then any injective $H$-module induces up to an injective $G$-module.  A particularly useful converse to this, which we will exploit in this paper, is 
if $H$ is a finite subgroup scheme of $G$, then inducing the trivial $H$-module $k$ up to $G$ produces an injective $G$-module if and only if $k$ is injective over $H$.  Induced modules of this form, when $G$ is semi-simple and $H$ is a finite Chevalley subgroup, have featured prominently in the work of Bendel, Nakano, and Pillen (see, for example, \cite[\S 2.1]{BNP3}). 

The paper is organized as follows. In the following section (Section 2), for arbitrary affine group schemes $G$ we generalize the BNP-induced modules by inducing from finite subgroup schemes 
to construct mock injective modules. With these constructions, it is shown when proper mock injective modules arise. This leads us to give a necessary and sufficient condition on $G$ 
(cf. Theorem~\ref{thm:main}) for the existence of proper mock injective modules. Our investigation leads to several natural questions which are presented in this section. Finally, we demonstrate that our 
constructions (for mock injectives) in this section can be used to show that a formation of the known Parshall Conjecture (cf. \cite{LN, D}) for infinitely generated modules is false. 

In Section 3, we explore the situation when a given mock injective module has finitely many composition factors in its socle. In the case when $G$ is reductive, we construct 
a proper mock injective with the socle isomorphic to the trivial module. On the other hand, using general results about mock injectives for parabolics, we prove that for 
a Borel subgroup of a reductive group, any mock injective with finite-dimensional socle must be injective. Finally we prove that for $G$ reductive, a mock injective $G$-module 
with a simple $G$-socle is injective over $G$ if and only if it has a good filtration.

\subsection{Acknowledgements} The authors would like to thank Eric Friedlander for initiating the interesting study of mock injectivity which inspired many of our results, and for providing helpful feedback on this paper.  We also thank Chris Bendel for carefully reading a previous draft of this article, and for the many good suggestions he offered.

\section{Conditions on $G$ for the Existence of Proper Mock Injective Modules}\label{S:existence}

\subsection{} Let $G$ be an arbitrary affine group scheme and $k$ be an algebraically closed field. The coordinate algebra $k[G]$, which can be realized as $\text{ind}_1^G k$, is 
an important example of an induced module for $G$.  More generally, we find that induction from finite subgroup schemes of $G$ can produce $G$-modules with interesting properties. Such 
subgroup schemes are always exact in $G$ \cite[I.5.13(b)]{J}.

\begin{prop}\label{prop:chevalley-induction} Let $H$ be a finite subgroup scheme of $G$.
\begin{itemize}
\item[(a)] $\textup{ind}_H^G \, k$ is an injective $G$-module if and only if $k$ is an injective $H$-module.
\item[(b)] If the Frobenius morphism $F:G \rightarrow G$ restricts to an automorphism of $H$, then $\textup{ind}_H^G \, k$ is injective over $G_r$ for every $r>0$.
\end{itemize}
\end{prop}

\begin{proof}
Since $H$ is exact in $G$, we may apply generalized Frobenius reciprocity \cite[Corollary I.4.6]{J} to obtain isomorphisms
\[
\text{Ext}_G^i(k,\textup{ind}_H^G \, k) \cong \text{Ext}_H^i(k,k)
\]
for every $i \ge 0$.  If $\textup{ind}_H^G \, k$ is injective over $G$, then the groups on the left side of the isomorphism are $0$ for every $i > 0$.  This implies that the cohomological variety of the finite group scheme $H$ is trivial, hence that $k$ is an injective $H$-module.  Conversely, if $k$ is injective over $H$, then $\textup{ind}_H^G \, k$ is injective over $G$, since induction always takes injective modules to injective modules \cite[Proposition I.3.9]{J}.  This proves the first statement.

Let $r>0$, and let $I$ be an injective $G$-module.  Then $I$ is injective over $H$.  Furthermore, if $F$ restricts to an automorphism of $H$, then $I^{(r)}$ is also injective over $H$.  By the tensor identity, we have
\[
I^{(r)} \otimes \textup{ind}_H^G \, k \cong \textup{ind}_H^G \, I^{(r)}.
\]
The module $\textup{ind}_H^G \, I^{(r)}$ is, by earlier remarks, injective over $G$, hence is injective over $G_r$.  On the other hand, over $G_r$,  $\textup{ind}_H^G \, k$ is a summand of $\textup{ind}_H^G \, I^{(r)}$, as is evident from the left hand side of the isomorphism above.  This shows that $\textup{ind}_H^G \, k$ is injective over $G_r$.
\end{proof}

If $H$ has the property that every simple $H$-module comes from a $G$-module, then we can also completely determine which $G$-modules will tensor with $\text{ind}_H^G \, k$ to produce injective $G$-modules.  This condition is satisfied whenever $H$ is unipotent, and also when $H$ is a Frobenius kernel or finite Chevalley subgroup of a semi-simple and simply connected group.

\begin{prop}
Let $H$ be a finite subgroup scheme of $G$ for which every simple $H$-module comes from a $G$-module.  Then for any $G$-module $M$, the module $M \otimes \textup{ind}_H^G \, k$ is injective over $G$ if and only if $M$ is injective over $H$.
\end{prop}

\begin{proof}
By the tensor identity,
$$M \otimes \text{ind}_H^G \, k \cong \text{ind}_H^G \, M.$$
If $M$ is injective over $H$, then $\textup{ind}_H^G \, M$ is injective over $G$.  Conversely, if $\textup{ind}_H^G \, M$ is injective over $G$, then for every $G$-module $N$, we find that
$$0 = \text{Ext}_G^i(N,\textup{ind}_H^G \, M) \cong \text{Ext}_H^i(N,M).$$
Since every simple $H$-module comes from some $G$-module, this immediately implies that $M$ is injective over the finite group scheme $H$.
\end{proof}

\subsection{} By using these induced modules from the previous subsection, we are now in position to give a complete picture of which affine algebraic groups over $k$ have proper mock injective modules.  Recall that $G^0$ denotes the connected component of $G$.

\begin{theorem}\label{thm:main}
Let $G$ be an algebraic group over $k$ which is defined and split over a finite subfield $\mathbb{F}_p$ of $k$.  Then 
$G$ has proper mock injective modules if and only if either $G^0$ is not a torus or $G/G^0$ has order divisible by $p$.
\end{theorem}

\begin{proof} If $p$ does not divide the order of $G(\mathbb{F}_q)$, then it must be the case that $G^0$ is a torus and that $G/G^0$ is a finite group of order not divisible by $p$.  To see this, we note that the second statement is clear, while the first follows from the fact that every element in $G(\mathbb{F}_q)$ would have to be semi-simple, since unipotent elements have order which is a power of $p$.  In this case, $G$ is linearly reductive \cite{Na}, hence there are no non-injective $G$-modules of any kind, and therefore there cannot be any proper mock injective $G$-modules.

If $p$ divides the order of $G(\mathbb{F}_q)$, then $k$ is a non-injective $G(\mathbb{F}_q)$-module.  By Proposition \ref{prop:chevalley-induction}, the $G$-module $\textup{ind}_{G(\mathbb{F}_q)}^G \, k$ is both non-injective over $G$, while also being injective over $G_r$ for all $r$.  Therefore, it is a proper mock injective $G$-module.
\end{proof}

\begin{remark} If $G$ is connected reductive and not a torus then the mock injective modules described above do not have a good filtration. One 
can see this because 
$$\textup{Ext}^{i}_{G}(k,\textup{ind}_{G(\mathbb{F}_q)}^G \, k)\cong \textup{Ext}^{i}_{G({\mathbb F}_{q})}(k,k)\neq 0$$ 
for some $i>0$. 
Here, the trivial module should be viewed as the Weyl module of highest weight $0$. 
\end{remark}

\subsection{} We will now show that it is even easier to exhibit non-injective $G$-modules which are mock injective with respect to the fixed-point subgroups of powers of the Frobenius morphism.

\begin{prop}
Let $G$ be an affine algebraic group defined over $\mathbb{F}_p$, $I$ an injective $G$-module, and $r > 0$.  Then $I^{(r)}$ is injective over $G(\mathbb{F}_q)$, for $q=p^s$ for all $s \ge 0$, but is not injective over $G$.
\end{prop}

\begin{proof}
Since  $G(\mathbb{F}_q)$ is exact in $G$, $I$ is injective on restriction to $G(\mathbb{F}_q)$.  The Frobenius map defines an automorphism of $G(\mathbb{F}_q)$, hence $I^{(r)}$ is also injective over it.  Finally, this module is trivial for $G_r$, thus cannot be injective over $G$.
\end{proof}

\subsection{} A natural question at this point is whether or not a non-injective $G$-module can be injective over all Frobenius kernels and finite Chevalley subgroups.  We at least observe that taking a tensor product of modules from our two classes of examples above cannot produce such an example.  That is, if $H$ is a finite subgroup scheme such that the Frobenius morphism restricts to an automorphism of $H$, then $I^{(r)}$ is injective over $H$, and so we have 
$$I^{(r)} \otimes \textup{ind}_H^G k \cong \textup{ind}_H^G \, I^{(r)},$$
which is then injective over $G$. In fact, one can go a step further and ask the following question.

\begin{quest}
Let $G$ be a reductive group. If $M$ is a $G$-module which is injective over every finite subgroup scheme of $G$, is $M$ then injective over $G$?
\end{quest}

For non-reductive groups, Friedlander's argument can be immediately used to show that the answer to this question is ``no.''  That is, embed such a group $G$ into some $GL_n$.  The coordinate algebra of $GL_n$ is then non-injective for $G$, but is injective over every finite subgroup scheme since they are all exact in $GL_n$. 

\subsection{Parshall Conjecture and infinite-dimensional modules} In this subsection we assume that $G$ is a reductive algebraic group scheme over a field $k$ of any positive characteristic. 
In 1986, Parshall \cite{P} conjectured that a finite-dimensional rational $G$-module which is projective over $G_1$ is then projective over 
$G(\mathbb{F}_p)$. This conjecture was proven by Lin and Nakano (cf. \cite{LN}).  Drupieski \cite{D} later generalized this result by proving that, 
for any $r \ge 1$, a finite-dimensional rational $G$-module that is projective over $G_r$ is projective over $G(\mathbb{F}_{q})$. 

Now one could also ask what happens to the statement of Parshall's Conjecture when the finite-dimensionality assumption is removed. 
As in the case with mock injectives, infinite-dimensional rational $G$-modules can exhibit new and interesting structures. 
We will show the natural analogues to the above conjectures are false by using the infinite-dimensional rational $G$-module $M$ 
which is projective over $G_r$ (as defined in Section~\ref{S:existence}) and by demonstrating that this module is not projective over $G(\mathbb{F}_{q})$. 

\begin{prop}
Let $r \ge 1$, and $q=p^s$ where $s \ge 1$.  Then the $G$-module $\textup{ind}_{G(\mathbb{F}_p)}^G \, k$ is projective over $G_r$, but is not projective over $G(\mathbb{F}_{q})$.
\end{prop}

\begin{proof}  Let $M = \textup{ind}_{G(\mathbb{F}_p)}^G \, k$. Proposition~\ref{prop:chevalley-induction} estabilishes the mock injectivity of $M$, so it remains only to show that $M$ is not 
projective over $G(\mathbb{F}_{q})$.  Since $kG(\mathbb{F}_{q})$ is a free $kG(\mathbb{F}_{p})$-module, every injective $G(\mathbb{F}_q)$-module is injective over $G(\mathbb{F}_p)$, thus it suffices to show that $M$ is not projective over $G(\mathbb{F}_{p})$.

The subgroup $G(\mathbb{F}_p)$ acts on $G$ via the right regular action.  Then $M \cong k[G]^{G(\mathbb{F}_p)}$, where
\[
k[G]^{G(\mathbb{F}_p)} = \{f \in k[G] \mid f(gh) = f(g) \text{ for all $g \in G(k), h \in G(\mathbb{F}_p)$} \}.
\] 
This allows us to define a $G(\mathbb{F}_p)$-module homomorphism $\text{ev}:M \rightarrow k$ given by $\text{ev}(f) = f(1)$ for any $f \in M$. 
Moreover, we also have a $G$-module homomorphism $\psi: k \rightarrow M$ given by $\psi(a) = a$ (the constant function) for any $a \in k$. 
The composition $\text{ev}\circ \psi = 1_{k}$, and hence there exists a $G(\mathbb{F}_p)$ decomposition $M = k \oplus N$ for some 
$G(\mathbb{F}_p)$-module $N$.  It follows immediately that $M$ is not projective over $G(\mathbb{F}_p)$.

\end{proof}

\section{Mock Injective Modules With Finite Socles}

\subsection{} Throughout this section, let $G$ be a connected reductive algebraic scheme group over $k$. 
We investigate the relationship between mock injective modules and injective modules when we impose additional finiteness conditions on our modules.  Specifically, we are concerned in this section with modules that have finite socles.  This inquiry was motivated by Friedlander's observation that for a unipotent group $U$, there are proper mock injective $U$-modules that embed in $k[U]$, 
hence have a trivial socle.  Thus we ask whether there exist proper mock injective modules with a simple socle, or at least a finite socle, over other types of algebraic groups.  

We find that there do exist proper mock injective modules for reductive groups which have a simple socle by producing an example of one for $SL_2$.  However, for the Borel subgroups of a reductive group, there do not even exist proper mock injective modules with a finite socle.  This is established by proving a more general result for mock injective modules of a parabolic subgroup $P_J$ of $G$ which are injective over the Levi factors of $P_J$.

\subsection{Unipotent Groups}
Let $U$ denote a connected unipotent group over $k$ which is defined over $\mathbb{F}_p$.
We give an alternate proof here to Friedlander's result showing that unipotent groups can have lots of proper mock injective modules with finite socle.
\begin{prop}\label{prop:unipotent}
Let $r\geq 1$, and $q=p^r$.  Then the proper mock injective $U$-module $M_r = \textup{ind}_{U(\mathbb{F}_{q})}^U\, k$ satisfies $\textup{soc}_U(M_r) = k$. 
\end{prop}
\begin{proof} The only simple $U$-module is the trivial module $k$. Now by Frobenius reciprocity, $\text{Hom}_{U}(k,M_{r})\cong \text{Hom}_{U({\mathbb F}_{q})}(k,k).$ 
It follows that $\text{soc}_{U} M_{r}\cong k$. 
\end{proof}

\subsection{Actions of the center} We begin by recalling the following fact about actions involving the center of a connected reductive algebraic group.  Note that the center in this case consists of semi-simple elements in $G$ 
(cf. \cite[Corollary 7.6.4(iii)]{S}).

\begin{lemma}
Let $Z=Z(G)$ be the center of a reductive algebraic group $G$, $X(Z)$ the character group of $Z$, and $M$ be any rational $G$-module.  There is a $G$-module decomposition
\[
M = \bigoplus_{\chi \in X(Z)} M_{\chi}
\]
where $M_{\chi} = \{ m \in M \mid z\cdot m = \chi(z)m\,\,\text{for all $z \in Z$}\}$. 

Moreover, if $\chi, \chi' \in X(Z)$ satisfies $\chi \neq \chi'$ and $M$, $N$ are $G$-modules satisfying $M=M_{\chi}$ and $N = N_{\chi'}$, then
$\textup{Hom}_G(M,N) = 0.$
\end{lemma}

Fix a maximal torus $T \le G$, and a Borel subgroup $B$ containing $T$.  We choose a set set of simple roots $\Delta$ so that the root subgroups contained in $B$ correspond to the negative roots.  The choice of $\Delta$ also determines the set of dominant weights $X(T)_+$.

For $J \subseteq \Delta$, let $P_J$ be the corresponding standard parabolic subgroup of $G$ containing $B$, with unipotent radical $U_J$ and Levi factor $L_J$. With $Z_J = Z(L_J)$ denoting the center of $L_J$, it can be verified that 
\[
X(Z_J) = X(T)/\mathbb{Z}J. 
\]
Thus, every central character is of the form $\chi = \lambda + \mathbb{Z}J$. If we let $\pi: X(T) \rightarrow X(Z_J)$ denote
the quotient map, then $\pi(\mathbb{Z}\Phi) = \mathbb{Z}\Phi/\mathbb{Z}J \cong \mathbb{Z}I$ where $I = \Delta \backslash J$. 
It follows that any $L_J$-module $M$ whose weights all lie in $\mathbb{Z}\Phi$, has a central character decomposition of the form 
\[
M = \bigoplus_{\chi \in \mathbb{Z}I} M_{\chi}.  
\]
This statement particularly  applies to $k[U_J]$.

\begin{lemma}\label{lemma:coordinate-decomposition}
The $L_J$-module $k[U_J]$ has a central character decomposition of the form
\[
k[U_J] = \bigoplus_{\chi \in \mathbb{N}I} k[U_J]_{\chi}
\]
where $\dim(k[U_J]_{\chi}) < \infty$ for all $\chi$. 
\end{lemma}
\begin{proof}
First, recall that $k[U_J]$ is a polynomial ring given by
\[
k[U_J] = k[x_{\gamma} | \gamma \in \Phi^+\backslash \Phi_J^+]
\]
where each $x_{\gamma}$ has weight $\gamma$. For any weight $\lambda \in X(T)$, the weight space 
$k[U_J]_{\lambda}$ is spanned by monomials of the form 
$\prod_{\gamma \in \Phi^+\backslash \Phi_J^+} x_{\gamma}^{n_{\gamma}}$ where 
$\lambda = \sum_{\gamma} n_{\gamma}\gamma$ with $n_{\gamma}\geq 0$ is any expression for $\lambda$. Since there
are only finitely many ways to express $\lambda$ as a non-negative integral combination of elements in $\Phi^+\backslash \Phi_J^+$,
we can conclude that the weight spaces are finite-dimensional. 

Fix a central character $\chi \in X(Z_J)$ which acts non-trivially on $k[U_J]$.
By the remark preceding this lemma and the above description of the weight spaces of $k[U_J]$, it follows that 
there exists a unique weight $\lambda \in \mathbb{N}I$ satisfying
\[
\chi = \lambda + \mathbb{Z}J.
\]
Hence, 
\[
k[U_J]_{\chi} = \bigoplus_{\mu \in \mathbb{Z}J}k[U_J]_{\lambda+ \mu}.
\]
So it will be sufficient to show that for any $\lambda \in \mathbb{Z}I$, there are only finitely many $\mu \in \mathbb{Z}J$ such that 
$k[U_J]_{\lambda + \mu} \neq 0$. In fact, it can be immediately deduced that any such $\mu$ must satisfy $\mu \in \mathbb{N}J$
(i.e. it must be a non-negative linear combination of elements in $\Phi_J^+$). 

Observe now that we can write 
\[
\Phi^+\backslash \Phi_J^+ = \Phi_I^+ \cup ((\Phi_I^+ + \Phi_J^+)\cap \Phi^+). 
\]
So $k[U_J]_{\lambda+\mu} \neq 0$ if and only if we can write 
\begin{align*}
\lambda + \mu &= \sum_{\alpha \in \Phi_I^+} r_{\alpha}\alpha + 
	\sum_{\alpha + \beta \in (\Phi_I^+ +\Phi_J^+)\cap \Phi^+} t_{\alpha+\beta}(\alpha+\beta) \\
	&= \sum_{\alpha \in \Phi_I^+} \left( r_{\alpha} + \sum_{\{\beta \in \Phi_J^+ \mid \alpha + \beta \in \Phi^+\}} t_{\alpha+\beta}\right) 
		\alpha + 
	    \sum_{\beta \in \Phi_J^+} \left(\sum_{\{\alpha \in \Phi_I^+ \mid \alpha+\beta \in \Phi^+\}} t_{\alpha+\beta}\right) \beta.
\end{align*}
where for all $\alpha \in \Phi_I^+$ and $\beta \in \Phi_J^+$,  
\begin{align*}
n_{\alpha} &= r_{\alpha} + \sum_{\{\beta \in \Phi_J^+ \mid \alpha + \beta \in \Phi^+\}} t_{\alpha+\beta} \geq 0\\
m_{\beta} &=   \sum_{\{\alpha \in \Phi_I^+ \mid \alpha+\beta \in \Phi^+\}}t_{\alpha + \beta} \geq 0
\end{align*}
and $\lambda = \sum n_{\alpha}\alpha$ and $\mu = \sum m_{\beta}\beta$. 
Since the $r_{\alpha}$ and $t_{\alpha + \beta}$ are also required to be non-negative integers, it follows that there are only finitely many $\mu \in \mathbb{N}J^+$ for which this is possible. Therefore,
$k[U_J]_{\chi}$ is finite-dimensional. 
\end{proof}

\subsection{} We begin this section with a standard fact about composition factor multiplicities.  For $\lambda \in X(T)_+$, $L(\lambda)$ is the simple $G$-module with highest weight $\lambda$, $H^0(\lambda) = \text{ind}_B^G \, \lambda$, and $I(\lambda)$ is the $G$-injective hull of $L(\lambda)$.

\begin{lemma}\label{lemma:finite-multiplicity}
Let $M$ be a finite-dimensional $G$-module, $M^*$ its linear dual, and $\lambda, \mu  \in X(T)_+$. Then 
\[
\dim \textup{Hom}_G(L(\mu), I(\lambda)\otimes M)  = [L(\mu)\otimes M^*:L(\lambda)] < \infty. 
\]
Consequently, every indecomposable summand of $I(\lambda)\otimes M$ occurs with finite multiplicity. 
\end{lemma}
\begin{proof}
By adjointness,
\[
\textup{Hom}_G(L(\mu), I(\lambda)\otimes M) = \textup{Hom}_G(L(\mu)\otimes M^*, I(\lambda)).
\]
The injectivity of $I(\lambda)$ implies that the functor $\textup{Hom}_G(-, I(\lambda))$ is exact on short exact sequences. 
Thus, $\dim \textup{Hom}_G(L(\mu)\otimes M^*, I(\lambda))$ counts the number of times the composition 
factor $L(\lambda)$ appears in $L(\mu)\otimes M^*$ (i.e., $[L(\mu)\otimes M^*:L(\lambda)]$). 
\end{proof}

\begin{remark}
Although not necessary for our main result, one might wonder if the socle of $I(\lambda) \otimes M$ is always finite-dimensional for a 
finite-dimensional module $M$. This would be equivalent to showing that $[L(\mu)\otimes M^*:L(\lambda)]=0$ for all but finitely many 
$\mu \in X(T)_+$. It turns out this is false in general, as one can see by taking $G=SL_2(\overline{\mathbb{F}_2})$, and 
setting $\lambda = 0$ and $M = L(1)$. Then it can verified that there are infinitely many $\mu \in X(T)_+ = \mathbb{N}$ such that 
 $[L(\mu)\otimes L(1):L(0)] \neq 0$, and thus $\text{soc}_G(I(0)\otimes L(1))$ is infinite-dimensional. 
\end{remark}

For any weight $\lambda \in X(T)$ which is dominant for $L_J$, 
the injective hull for $L_{P_J}(\lambda)$ (i.e., the inflation of $L_{J}(\lambda)$ to $P_{J}$) is given by
\[
I_{P_J}(\lambda) = \text{ind}_{L_{J}}^{P_{J}}\ I_{L_{J}}(\lambda)\cong  I_{L_J}(\lambda)\otimes k[U_J]
\]
where $L_J$  acts by conjugation on $k[U_J]$. This can be verified by using the fact that induction take injectives to injectives and 
Frobenius reciprocity to show that the socle of $I_{P_J}(\lambda)$ is $L_{P_J}(\lambda)$.

We would like to show that all of the indecomposable 
summands of $I_{P_J}(\lambda)|_{L_J}$ occur with finite multiplicity. Unfortunately, the module $k[U_J]$ is not
a finite-dimensional $L_J$-module, so we cannot immediately apply Lemma~\ref{lemma:finite-multiplicity}. We can remedy 
this situation with the use of the central characters for $L_{J}$. 

\begin{prop}\label{prop:injective-finite}
If $I$ is any injective $P_J$-module with a finite-dimensional socle, then 
\[
I|_{L_J} = \bigoplus_{\nu \in X(T)} I_{L_J}(\nu)^{n_{\nu}}
\]
where $n_{\nu} < \infty$ for all $\nu \in X(T)$. 
\end{prop}
\begin{proof}
Without loss of generality, one can assume that $I = I_{P_J}(\lambda)$ for some $\lambda \in X(T)$ and hence 
$I = I_{L_J}(\lambda) \otimes k[U_J]$. It will be sufficient to show that for any weight $\mu \in X(T)$ dominant for $L_J$,
$$\dim \text{Hom}_{L_J}(L_J(\mu), I_{L_J}(\lambda)\otimes k[U_J]) < \infty.$$ 

Since $I_{L_J}(\lambda)$ and $L_J(\mu)$ are indecomposable $L_J$-modules, then $Z_J$ acts on the former by 
the character $[\lambda]= \lambda + \mathbb{Z}J$ and on the latter by $[\mu] = \mu + \mathbb{Z}J$. Therefore, 
\[
\text{Hom}_{L_J}(L_J(\mu), I_{L_J}(\lambda)\otimes k[U_J]) = 
	\text{Hom}_{L_J}(L_J(\mu), I_{L_J}(\lambda)\otimes k[U_J]_{\mu-\lambda})
\]
because any $L_J$-homomorphism must preserve the $X(Z_J)$ weight spaces. 
The proposition immediately follows from  Lemmas~\ref{lemma:coordinate-decomposition} and \ref{lemma:finite-multiplicity}.
\end{proof}

\subsection{Parabolic subgroups}  Let $F:P_{J}\rightarrow P_{J}$ be the Frobenius morphism and $(P_J)_rL_J=(F^{r})^{-1}(L_{J})$. 
The following result investigates $P_{J}$-modules which are injective over $(P_J)_rL_J=(F^{r})^{-1}(L_{J})$ and have finite socles. 

\begin{theorem}\label{thm:parabolic}
If $M$ is a $P_J$-module with a finite-dimensional socle which is injective over $(P_J)_rL_J$ for all $r\geq 1$, then $M$ is injective over $P_J$. 
\end{theorem}

\begin{proof}
It can be deduced from the finite-dimensionality of $\text{soc}_{P_J}(M)$ that the $P_J$-injective hull of $M$ is of the form 
$I\otimes k[U_J]$ where $I$ is an injective $L_J$-module with a finite-dimensional socle, and we get an embedding 
\[
\iota: M \hookrightarrow I\otimes k[U_J].
\]
By Proposition~\ref{prop:injective-finite}, 
\[
(I\otimes k[U_J])|_{L_J} = \bigoplus_{\nu \in X(T)} I_{L_J}(\nu)^{n_{\nu}}
\]
with $n_{\nu} < \infty$ for all $\nu$. 

By assumption, $M|_{(P_J)_rL_J}$ is injective for all $r$, and thus there exist embeddings over $L_J$ (or $(P_J)_rL_J$) of the form
$\phi_r: I\otimes k[(U_J)_r] \hookrightarrow M$. For each $r$, let $M_r = \text{im}(\phi_r) \subseteq M \subseteq I\otimes k[U_J]$.
It will be sufficient to show that $I\otimes k[U_J] = \bigcup_r M_r$. 

The explicit descriptions of the coordinate algebras for $U$ and $U_{r}$ can be used to show that  $k[U_J] = \bigcup_r k[(U_J)_r]$ where each $k[(U_J)_r]$ can be canonically embedded inside $k[U_J]$ as a
$(P_J)_rL_J$-module. We can deduce that  $I\otimes k[U_J] = \bigcup_r I\otimes k[(U_J)_r]$. Thus, for any $\nu$ satisfying $n_{\nu} > 0$, 
there must exist some $r\geq 1$ such that $I_{L_J}(\nu)^{n_{\nu}} \mid M_r\subseteq M$. This forces $M = I\otimes k[U_J]$ 
 and therefore $M$ is injective as a $P_J$-module. 
\end{proof}

\subsection{Borel subgroups} The aforementioned theorem in the previous section has a particularly nice formulation in the case when $P_{J}=B$. 

\begin{cor}\label{prop:borel}
Let $M$ be any mock injective $B$-module which has a finite-dimensional socle, then $M$ is an injective $B$-module. 
\end{cor}
\begin{proof}
In this situation, $B= P_{\emptyset}$ and $T=L_{\emptyset}$. The proof then follows by applying Theorem~\ref{thm:parabolic} 
and observing that a module $M$ is injective over $B_rT$ if and only if it is injective over $B_r$. 
\end{proof}

\subsection{Reductive groups} In this section we show that the analogue of Corollary~\ref{prop:borel} does not hold for reductive groups. 
Specifically, we will give an example of a module for $G=SL_2(k)$, with $k=\overline{\mathbb{F}_2}$, which 
is proper mock injective and has a simple socle. Let 
\[
W = \left \{ \begin{pmatrix} 
            1 & 0 \\
            0 & 1
            \end{pmatrix}, \,
            \begin{pmatrix}
            0 & 1 \\
            1 & 0
            \end{pmatrix}
           \right \} \subseteq SL_2(\mathbb{F}_2)
 \]
 denote the Weyl group. Since $2 \mid |W|$, then there exist non-trivial extensions of the trivial representation for $W$. 
 This implies that the module $\textup{ind}_W^G\, k = k[G]^W$ is proper mock injective by the same argument as in 
 Theorem~\ref{thm:main}. 
 Furthermore, the normalizer of the torus, $N_G(T) = WT$, is given by
 \[
 WT = \left \{ \begin{pmatrix} 
            a & 0 \\
            0 & a^{-1}
            \end{pmatrix}, \,
            \begin{pmatrix}
            0 & a \\
            a^{-1} & 0
            \end{pmatrix} \mid a \in k^*
           \right \}.
 \]
 
The following theorem demonstrates that $M = \textup{ind}_{WT}^G k$ is proper mock injective and has a simple socle. 

\begin{prop} Assume that the characteristic of $k$ is $2$. Let $G=SL_2$, and let $M=\textup{ind}_{WT}^G\, k$.  Then 
\begin{itemize}
\item[(a)] $\textup{soc}_G M = k$, 
\item[(b)] $M$ is proper mock injective for $G$. 
\end{itemize} 
\end{prop}
\begin{proof} (a) In this situation, one can identify $X(T) = \mathbb{Z}$ and $\mathbb{Z}[X(T)] = \mathbb{Z}[q,q^{-1}]$.  We will first show by induction that the simple module
 $L(\lambda)$ does not contain a vector of weight $0$ for any $\lambda > 0$. For the base case, observe that 
 \[
 \text{ch } L(1) = q + q^{-1}
 \]
 and thus doesn't contain a $0$-weight vector. Now suppose that $\lambda \geq 2$ and that the result holds for all $\mu < \lambda$. 
 If $\lambda$ is odd, then all the weights appearing in $L(\lambda)$ must also be odd, thus $0$ cannot appear as a weight. Conversely,
 if $\lambda$ is even, then by Steinberg's tensor product theorem $L(\lambda) = L(\lambda/2)^{(1)}$ and we can apply
 the inductive hypothesis to $\lambda/2 < \lambda$. 
 
 Therefore, (a) follows because for any $\lambda > 0$, 
 \[
 \text{Hom}_G(L(\lambda), M) \cong \text{Hom}_{WT}(L(\lambda), k) \subseteq \text{Hom}_T(L(\lambda),k) = 0
 \]
 and 
 \[
 \text{Hom}_G(k, M) \cong \text{Hom}_{WT}(k, k)\cong k.
 \]

(b) Next we show that $M$ is proper mock injective. Since $WT/T \cong W$, it follows from a Lyndon-Hochschild-Serre spectral sequence argument that 
for some $i > 0$,
$$\text{Ext}^i_G(k,M) \cong \text{Ext}^i_{WT}(k,k) \cong   \text{H}^i(W, k^T) = \text{Ext}^i_W(k,k) \neq 0.$$
Therefore, $M$ is not injective, and in fact it doesn't even have a good filtration. 

It remains to prove that $M$ is mock injective, this will follow by showing that $M$ is a summand
of the mock injective module  $\text{ind}_W^G\, k$. To see why this is true, first observe that 
\[
\text{ind}_W^G\, k = \text{ind}_{WT}^G\, \text{ind}_{W}^{WT}\, k.
\]
Let $N = \text{ind}_{W}^{WT}\, k$.  It can be verified that $N = k[T]$ where $T$ acts as the left regular representation and $W \cong WT/T$
acts by conjugation. Now $k[T]$ is an injective $T$-module and the zero weight space under this action appears exactly once. Furthermore, the zero weight 
space is a summand of $k[T]$ under $W$.  Thus we have $N \cong k \oplus S$ for some $WT$-module $S$.  It follows that $\text{ind}_{W}^G\, k = M \oplus \text{ind}_{WT}^G\, S$, and therefore that $M$ is mock injective. 
\end{proof}

\subsection{Mock injectives with good filtrations}
As exhibited above, for $G$ reductive, a proper mock injective $G$-module may have a simple socle.  We will now show that this cannot occur if we also assume that $M$ has a good filtration.

Let $M$ be a mock injective $G$-module with simple $G$-socle $L(\lambda)$, and choose $r_0$ such that $\lambda \in X_{r_0}(T)$.  For any $r \ge r_0$ it follows that the simple $G_r$-module $L_r(\lambda)$ is the only iso-class appearing in $\text{soc}_{G_r}(M)$, and the only iso-class appearing in $\text{soc}_{G_r} \, I(\lambda)$.  Let $Q_r(\lambda)$ denote the $G_r$-injective hull of $L_r(\lambda)$.  As a consequence of the previous remarks, the only $G_r$-summands appearing in each of these modules is $Q_r(\lambda)$.  Since $M$ is injective over $G_r$, there is a splitting $I(\lambda) \rightarrow M$.  This shows that $I(\lambda)/M$ is isomorphic over $G_r$ to a direct sum of the modules $Q_r(\lambda)$, hence that $L_r(\lambda)$ is the only iso-class appearing in $\text{soc}_{G_r}(I(\lambda)/M)$.  Now this holds for all $r \ge r_0$, so if $L(\mu)$ appears in $\text{soc}_G(I(\lambda)/M)$ for some $\mu$, we can pick $r$ large enough that $\mu \in X_r(T)$, in which case $L_r(\mu)$ appears in the $G_r$-socle of $I(\lambda)/M$.  Thus $L_r(\mu) \cong L_r(\lambda)$, and since $\mu,\lambda \in X_r(T)$, it follows that $\mu=\lambda$.  We have therefore proved:

\begin{lemma}\label{lemma:quotientsocle}
If $M$ is a mock injective $G$-module having simple socle $L(\lambda)$, then $L(\lambda)$ is the only iso-class of simple $G$-module appearing in $\text{soc}_G(I(\lambda)/M)$.
\end{lemma}

We recall that for any $G$-module $N$, the injectivity of $I(\lambda)$ and the short exact sequence
\[
0 \rightarrow M \rightarrow I(\lambda) \rightarrow I(\lambda)/M \rightarrow 0
\]
gives rise to the exact sequence
\[
0 \rightarrow \text{Hom}_G(N, M) \rightarrow \text{Hom}_G(N, I(\lambda)) \rightarrow 
\text{Hom}_G(N, I(\lambda)/M) \rightarrow \text{Ext}^1_G(N, M) \rightarrow 0.
\]
From this we can deduce the following result.  We recall that $V(\lambda)$ is the Weyl module with highest weight $\lambda$.

\begin{theorem}
If $M$ is a mock injective $G$-module with a simple socle, then $M$ is injective if and only if it has a good filtration, in which case $M \cong I(\lambda)$ for some $\lambda \in X(T)_+$.
\end{theorem}

\begin{proof}
The ``only if" direction is clear since injective modules always have a good filtration.  On the other hand, assume that $M$ has a good filtration, and embed $M$ into its injective hull.  Since $M$ has a simple socle, its injective hull is of the form $I(\lambda)$ for some $\lambda$.  By \cite[II.4.18]{J}, any good filtration of $I(\lambda)$ has $H^0(\lambda)$ appearing with multiplicty one.

For each $V(\mu)$ we have that $\text{Ext}^1_G(V(\mu), M) = 0$, giving a short exact sequence
\[
0 \rightarrow \text{Hom}_G(V(\mu), M) \rightarrow \text{Hom}_G(V(\mu), I(\lambda)) \rightarrow 
\text{Hom}_G(V(\mu), I(\lambda)/M) \rightarrow 0
\]
showing that $\dim_k (\text{Hom}_G(V(\mu), I(\lambda)/M)) < \infty$.  By \cite[II.4.17]{J}, $I(\lambda)/M$ has a good filtration.  If $I(\lambda)/M \ne 0$, then by Lemma~\ref{lemma:quotientsocle}, $L(\lambda) \subseteq I(\lambda)/M$.  Consequently, $H^0(\lambda) \subseteq I(\lambda)/M$ and $H^0(\lambda) \subseteq M$, so it occurs with multiplicity $2$ in a good filtration of $I(\lambda)$, a contradiction.  Thus, $I(\lambda)/M = 0$.
\end{proof}

Finally, although $M$ in general cannot be shown to have a good filtration (for this would contradict the example for $SL_2$ given earlier), we can at least say that $M$ starts out with the beginnings of a good filtration.

\begin{theorem}
Let $M$ be a mock injective $G$-module whose socle is isomorphic to $L(\lambda)$.  Then $H^0(\lambda)$ is isomorphic to a submodule of $M$.
\end{theorem}

\begin{proof}
By \cite[Lemma II.4.15]{J}, it suffices to show that
$$\text{Ext}_G^1(V(\mu),M) = 0$$
for all $\mu < \lambda$.  The sequence
\[
0 \rightarrow \text{Hom}_G(V(\mu), M) \rightarrow \text{Hom}_G(V(\mu), I(\lambda)) \rightarrow 
\text{Hom}_G(V(\mu), I(\lambda)/M) \rightarrow \text{Ext}^1_G(V(\mu), M) \rightarrow 0
\]
shows that $\text{Hom}_G(V(\mu), I(\lambda)/M)$ surjects onto $\text{Ext}^1_G(V(\mu), M)$.  However, $L(\lambda)$ is not a composition factor of $V(\mu)$ since $\lambda > \mu$.  By Lemma~\ref{lemma:quotientsocle}, any simple $G$-submodule of $I(\lambda)/M$ is isomorphic to $L(\lambda)$.  Therefore,  $\text{Hom}_G(V(\mu), I(\lambda)/M) = 0$, and so $\text{Ext}^1_G(V(\mu), M)$ must also be $0$.
\end{proof}


\begin{thebibliography}{88888888}

\bibitem[And1]{And1} H.H. Andersen,
Extensions of modules for algebraic groups, {\em Amer. J. Math.},
{\sf 106}, (1984), 498-504.

\bibitem[And2]{And2} H.H. Andersen,
Extensions of simple modules for finite Chevalley groups, {\em J. 
Algebra},
{\sf 111}, (1987), 388-403.

\bibitem[BNP1]{BNP1} C.P. Bendel, D.K. Nakano, C. Pillen, On comparing the cohomology 
of algebraic groups, finite Chevalley groups and Frobenius kernels, 
{\em J. Pure and Applied Algebra}, {\sf 163}, (2001), 119-146. 

\bibitem[BNP2]{BNP2} C.P. Bendel, D.K. Nakano, C. Pillen, Extensions 
for finite Chevalley groups II, {\em Transactions of the AMS}, {\sf 354}, (2002), 4421-4454. 

\bibitem[BNP3]{BNP3} C.P. Bendel, D.K. Nakano, C. Pillen, Extensions for 
finite Chevalley groups I, {\em Advances in Math.}, {\sf 183}, (2004), 380--408.  

\bibitem[BNP4]{BNP4} C. Bendel, D. Nakano, C. Pillen, Extensions for finite groups of Lie type II: Filtering the truncated induction functor, {\em Contemp. Math.}, {\sf 413}, (2006), 1-23.

\bibitem[BNP5]{BNP5} C.P. Bendel, D.K. Nakano, C. Pillen, Extensions for 
finite Chevalley groups III: Rational and Generic Cohomology, {\em Advances in Math.}, {\sf 262}, (2014), 484-519.  

\bibitem[CPS]{CPS} E. Cline, B.J. Parshall, L. Scott, On injective modules for infinitesimal algebraic groups I, 
{\em J. London Math. Soc.}, {\sf 31}, (1985), 277-291. 

\bibitem[CPSvdK]{CPSvdK} E. Cline, B. Parshall, L. Scott, W. van der Kallen,
Rational and generic cohomology, {\em Invent. Math.}, {\sf 39}, (1977), 143-163.

\bibitem[D]{D} C. Drupieski, On projective modules for Frobenius kernels and finite Chevalley groups, {\em Bull. London Math. Soc.}, {\sf 45}, (2013),
no. 4, 715-720. 

\bibitem[F1]{F1} E. Friedlander, Support varieties for rational representations, {\em Compos. Math.}, {\sf 151}, (2015), no. 4, 765-792.

\bibitem[F2]{F2} E. Friedlander, Filtrations, 1-parameter subgroups, and rational injectivity, ArXiv:1408.2918.

\bibitem[J]{J} J. Jantzen, {\em Representations of Algebraic Groups}, 2nd ed. Mathematical Surveys and Monographs, 107, American Mathematical Society, 2003.

\bibitem[LN]{LN} Z. Lin, D. Nakano, Complexity for modules over finite Chevalley groups and classical Lie algebras, {\em Invent. Math.}, {\sf 138}, (1999),
85-101. 

\bibitem[Na]{Na} M. Nagata, Complete reducibility of rational representations of a matric group, {\em J. Math. Kyoto Univ.}, {\sf 1},(1961/1962), 87-99.

\bibitem[N]{N} D.K. Nakano, Cohomology of algebraic groups, finite groups, and Lie algebras: Interactions and Connections, {\em Lie Theory and Representation Theory}, 
Survey of Modern Mathematics Volume II (edited by N. Hu, B. Shu, J.P. Wang), International Press, 2012, pp. 151-176. 

\bibitem[P]{P} B. Parshall, {\em Cohomology of algebraic groups}, The Arcata conference on representations of finite groups (Arcata, CA, 1986), 
Proceedings of Symposia in Pure Mathematics, {\sf 47}, American Mathematical Society, (1987), 233-248. 

\bibitem[S]{S} T. A. Springer, {\em Linear Algebraic Groups}, Progress in Math. {\sf 9}, Birkh\"{a}user, 1998.

\end{thebibliography}
\end{document}